\documentclass[12pt]{amsart}

\usepackage{pgf,tikz}
\usetikzlibrary[patterns]
\usepackage{mathrsfs}
\usetikzlibrary{arrows}
\usepackage{xcolor}
\usepackage{soul}

\usepackage{fancyhdr}
\usepackage{stmaryrd}
\usepackage[scale=0.75]{geometry}
\usepackage{mathtools}
\usepackage{color}
\usepackage{url}
\usepackage{amsmath}
\usepackage{amsthm}
\usepackage{amssymb}
\usepackage{amscd}          			

\usepackage{esint}

\usepackage{bbm}            			
\usepackage{mathrsfs}           		

\usepackage[latin1]{inputenc}
\usepackage[english]{babel}     		

\usepackage{graphicx}          		

\usepackage{verbatim}           		

\usepackage[bookmarksnumbered,colorlinks,plainpages,backref]{hyperref}

\numberwithin{equation}{section}		

\theoremstyle{plain}
\newtheorem{Theorem}{Theorem}[section]

\newtheorem{Proposition}[Theorem]{Proposition}
\newtheorem{Conjecture}[Theorem]{Conjecture}
\newtheorem{theorem}[Theorem]{Theorem}
\newtheorem{proposition}[Theorem]{Proposition}
\newtheorem{corollary}[Theorem]{Corollary}
\newtheorem{lemma}[Theorem]{Lemma}

\theoremstyle{definition}
\newtheorem{Definition}[Theorem]{Definition}
\newtheorem{example}[Theorem]{Example}
\newtheorem{definition}[Theorem]{Definition}

\theoremstyle{remark}

\newtheorem{remark}[Theorem]{Remark}

\newcounter{alphatheo}
\renewcommand\thealphatheo{\Alph{alphatheo}}
\newenvironment{alphatheo}{\refstepcounter{alphatheo}\medskip\noindent{\bf Theorem \thealphatheo.} \it }{\medskip}

\def\M{{\mathcal M}}
\newcommand{\C}{\mathbb{C}}       
\newcommand{\N}{\mathbb{N}}     
\newcommand{\R}{\mathbb{R}}     

\DeclareMathOperator{\diver}{div}			
\DeclareMathOperator*{\esssup}{ess\,sup}		
\DeclareMathOperator*{\essinf}{ess\,inf}		
\DeclareMathOperator*{\real}{Re}
\DeclareMathOperator*{\imag}{Im}

\usepackage[mathcal]{eucal}					

\def\eps{\epsilon}

\def\L{{{\mathcal{L}}}}

\def\ccinf{C^\infty_{c}}
\def\ds{\displaystyle}
\def\rt{\rightarrow}

\def\ccinf{C^\infty_{c}}
\def\cinf{C^\infty}

\def\<{\langle}
\def\>{\rangle}
\def\eps{\varepsilon}

\renewcommand{\geq}{\geqslant}
\renewcommand{\leq}{\leqslant}
\renewcommand{\epsilon}{\varepsilon}

\allowdisplaybreaks

\title[Solvability of elliptic linear eq with measure data in $L^p_{w}$]{Solvability of elliptic  homogeneous linear equations with measure data in weighted Lebesgue spaces}

\author {Victor Biliatto}
\address{Departamento de Computa\c{c}\~ao e Matem\'atica, Universidade de S\~ao Paulo, Ribeir\~ao Preto, SP, 14040-901, Brazil}
\email{victor.biliatto@usp.br}

\author {Joel Coacalle}
\address{Departamento de Computa\c{c}\~ao e Matem\'atica, Universidade de S\~ao Paulo, Ribeir\~ao Preto, SP, 14040-901, Brazil}
\email{jcoacalle@usp.br}

\author {Tiago Picon}
\address{Departamento de Computa\c{c}\~ao e Matem\'atica, Universidade de S\~ao Paulo, Ribeir\~ao Preto, SP, 14040-901, Brazil}
\email{picon@ffclrp.usp.br}

\thanks{{All the authors were supported by the S\~ao Paulo Research Foundation (FAPESP) grant numbers 2025/00433-1, 2024/20985-6, 2022/02211-8 and  2018/15484-7. The third author was supported by Conselho Nacional de Desenvolvimento Cient\'ifico e Tecnol\'ogico (CNPq) grant number 311430/2018-0.}}



\begin{document}

\begin{abstract} Let $A(D)$ be an elliptic  homogeneous linear differential operator {with complex constant coefficients}, $ \mu $ be a vector-valued Borel measure and $w$ be a positive locally integrable function on $ \R^N$. In this {work}, we present sufficient conditions on $\mu$ and $w$ for the existence of solutions in the weighted Lebesgue spaces $L^p_w$ for the equation $A^*(D)f=\mu$, for $ 1\leq p<\infty $. Those conditions are related to a certain control of the Riesz potential of the measure $\mu$. We also present sufficient conditions for the solvability {when $p=\infty$} adding a canceling condition on {the operator}. Our method is based on a new weighted $L^1$ Stein-Weiss type inequality on measures for a special class of vector fields.
\end{abstract}

\maketitle

\section{Introduction}\label{sect intro}

Phuc and Torres presented in \cite{PT} a characterization for the Lebesgue solvability, in distributional sense, of the equation
\begin{equation} \label{eq div} 
\diver f=\mu
\end{equation}
with measure data $\mu$. They proved in \cite[Theorems 3.1 and 3.2]{PT} that if $\mu$ is a non-negative locally finite Radon measure on $\R^N$, then \eqref{eq div} has a solution in {$L^p(\R^N, \R^N)$ for $1<p<\infty$} if and only if $\mu$ has finite $(1,p)-$energy, \emph{i.e.} $\|I_1 \mu\|_{L^p}<\infty$, where $I_1$ is the Riesz potential operator. In particular, if there is a solution in $L^p(\R^{N}, \R^N)$ for $1 \leq p \leq N/(N-1)$, then $\mu \equiv 0$. In the case for bounded vector fields, they proved in \cite[Theorem 3.3]{PT} that if $\mu$ is a non-negative locally finite Radon measure on $\R^N$, there exists a vector field $f \in L^\infty(\R^N,\R^N)$ satisfying \eqref{eq div} if, and only if, $\mu$ has the property
\begin{equation*} \label{ahlfors PT}
\mu(B(x,r)) \leq M r^{N-1}, \quad \text{for all } r>0 \text{ and } x \in \R^N,
\end{equation*}
for some positive constant $M$ independent of $x$ and $r$. A key element of their proof is the validity of the estimate
\begin{equation} \label{desig PT}
\left| \int_{\R^N} \varphi(x) \, d\mu \right| \leq C \int_{\R^N} |\nabla \varphi (x)| \, dx
\end{equation}
for all $\varphi \in \ccinf(\R^N)$, which is attained using the boxing inequality and the coarea formula for functions of bounded variation.

Some time later in \cite{MRT}, Moonens, Russ and Tuominen extended these results in the setting of vector fields for weighted Lebesgue spaces. They showed in \cite[Theorem 3.21]{MRT} that for a non-negative Radon measure $\mu$ on $\R^N$ and a 1-admissible weight $w$, there exists a vector field $f \in L^\infty_{1/w}(\R^N,\R^N)$ solving \eqref{eq div} if $\mu$ satisfies
\begin{equation} \label{ahlfors MRT}
\mu(B(x,r)) \leq \dfrac{M}{r} \int_{B(x,r)} w(x) \, dx, \quad \text{for all } r>0 \text{ and } x \in \R^N,
\end{equation}
for some positive constant $M$ independent of $x$ and $r$. For each positive $w \in L^{1}_{loc}(\R^{N})$, we recall that $L^\infty_{1/w}(\R^N,\R^N)$ is the space of measurable functions $f:\R^N \rt \R^N$ for which $\|f/w\|_{{L^\infty}} < \infty$, and $w$ is called a 1-admissible weight if it is doubling and 
supports a (1,1)-Poincar\'e inequality {(see \cite[{Definitions 2.1 and 2.2}]{MRT})}. Similar to the previous case, they used a weighted version of the boxing inequality and a coarea formula to show that \eqref{ahlfors MRT} implies the \emph{a priori} estimate
\begin{equation*} \label{desig MRT}
\int_{\R^N} |\varphi (x)| \, d\mu \leq C \int_{\R^N} |\nabla \varphi (x)| w(x) \, dx
\end{equation*}
for every compactly supported Lipschitz function $\varphi:\R^N \rt \R$. We point out that, when $w \equiv 1$, their result recovers one direction of \cite[Theorem 3.3]{PT}.

Recently, the first and third authors extended in \cite{BP} the results due to Phuc and Torres for the equation
\begin{equation} \label{eq A(D)} 
A^*(D)f=\mu,
\end{equation}
where $A(D): C^{\infty}(\R^N,E)\rightarrow C^{\infty}(\R^N,F)$ is a homogeneous linear differential operator {of order $1\leq m<N$} with complex constant coefficients \emph{i.e.} $A(D) = \sum_{|\alpha|=m} a_\alpha \partial^\alpha$,
where $E,F$ are finite dimensional complex vector spaces and the coefficients $a_\alpha \in \L(E,F)$ are linear transformations from $E$ to $F$ that do not depend on $x \in \R^N$. 
Here $A^*(D)$ is the (formal) adjoint operator of $A(D)$ and $\mu$ is a vector-valued complex measure on $\R^N$. {They proved in \cite[Theorem A]{BP} that if $1 \leq p \leq N/(N-m)$ and $f \in L^p(\R^N,F^{*})$ is a solution for \eqref{eq A(D)} when $\mu$ is {non-negative}, then $\mu \equiv 0$. In addition, if $N/(N-m) < p < \infty$ and  $f \in L^p(\R^N,F^{*})$ is a solution for \eqref{eq A(D)} then $\mu$ has finite $(m,p)-$energy. Conversely, if $|\mu|$ has finite $(m,p)-$energy and $A(D)$ is \emph{elliptic} then there exists a function $f \in L^p(\R^N,F^{*})$ solving \eqref{eq A(D)}}. The symbol $|\mu|$ denotes the total variation of the measure $\mu$ and we say a measure $\nu$ has finite $(m,p)-$energy if $\|I_m \nu\|_{L^p} < \infty$. 
We recall that the operator $A(D)$ is \emph{elliptic} if the symbol $A(\xi): E \rt F$ given by
\begin{equation}\nonumber
A(\xi) = \sum_{|\alpha|=m} a_\alpha \xi^\alpha
\end{equation}
is injective for all $\xi \in \R^{N}\setminus\left\{0\right\}$. In the proof, it is shown that if there exists a solution $f \in L^p(\R^N,F^*)$ for \eqref{eq A(D)}, then
\begin{equation}\label{eq:ImRalpha}
I_m \mu(x) = C \sum_{|\alpha|=m} a^{*}_\alpha \, R^\alpha f(x),
\end{equation}
where $R^\alpha$ is a composition of Riesz transforms, which are bounded in $L^p$  for $ 1<p<\infty $ and weak type for $p=1$. For bounded vector fields, it was proved the following result:

\begin{theorem}[{\cite[Theorem B]{BP}}]\label{theorem BP}
Let $A(D)$ be a homogeneous linear differential operator of order $1\leq m<N$ on $\R^N$ from $E$ to $F$
and $\mu$ be a Borel measure on $\R^N$ with values in $E^*$. If $A(D)$ is \emph{elliptic} and \emph{canceling}, and $\mu$ satisfies
\begin{equation}\label{ahlfors BP1}
|\mu|(B(0,r)) \leq M_1 r^{N-m}, \quad \text{for all } r>0,
\end{equation}
and the potential control
\begin{equation}\label{wolff BP1}
\ds\int_0^{{|y|/2}} \dfrac{{|\mu|(B(y,r))}}{r^{N-m+1}} \, dr \leq M_2, \quad  \text{ uniformly on } y ,
\end{equation}
then there exists $f \in L^\infty(\R^N,F^{*})$ solving \eqref{eq A(D)}.
\end{theorem}
An example of {non-negative} measure satisfying \eqref{ahlfors BP1} and \eqref{wolff BP1} is given by the weighted power $d\nu:=|x|^{-m}dx$ (see \cite[Remark 4.1]{BP}). We say that $A(D)$ is a \emph{canceling} operator when
\[\ds{{\bigcap_{\xi \in \R^{N}\setminus\left\{0\right\}}\,A(\xi)[E]=\left\{ 0 \right\}}}.\]
The theory of canceling operators is due to J. Van Schaftingen (see \cite{VS}), motivated by studies of some $L^1$ \emph{a priori} estimates for first order operators including the {div-curl} operators and chain complexes. We point out that $A(D)=\nabla$ is always elliptic and is canceling if and only if $N \geq 2$ and, {in this case}, $A^*(D)=\diver$, recovering \eqref{eq div}. Once again, the proof of the previous theorem follows from an \emph{a priori} estimate:
\begin{equation} \label{desig BP}
\left| \int_{\R^N} \varphi(x) \, d\mu \right| \leq C \int_{\R^N} |A(D) \varphi(x) | \, dx, \qquad  {\text{for all}} \,\,\varphi\in C^\infty_c(\R^N, E).
\end{equation}

In contrast to the proof of \eqref{desig PT}, which relies on the coarea formula specifically applied to the gradient operator, an alternative method was utilized to derive \eqref{desig BP}, leading to an improvement of the Stein-Weiss type inequality in $L^1$-norm stated in \cite[Lemma 3.2]{DNP} (see \cite[Lemma 4.3]{BP} for more details). 

In this work, {we continue the program started in \cite{BP} addressing} the solvability of equation \eqref{eq A(D)} for elliptic homogeneous differential operators in the weighted Lebesgue spaces. 
{As mentioned before, the method is to find} sufficient conditions on the measure $\mu$ and the weight $w$ for the \emph{a priori} estimate
\begin{equation*}\label{desig main}
	{\left|\int_{\R^N} \varphi(x) \, d\mu \right|} \leq C \int_{\R^N} |A(D) \varphi (x)| w(x) \, dx, \qquad {\text{for all}} \,\, \varphi\in C^\infty_c(\R^n, E),
\end{equation*}
{to hold}. 
{Our first result concerns the solvability of \eqref{eq A(D)} on weighted Lebesgue space $L^p_w$ with $1 < p < \infty$.}\\

{
\begin{alphatheo}\label{theoA}
Let $A(D)$ be a homogeneous linear differential operator of order $1 \leq m<N$ on $\R^{N}$, $N\ge2$, from $E$ to $F$, $\mu \in \M(\R^N,E^{*})$ and  {$w\in L^{1}_{loc}(\R^N)$} be a {positive} weight.
\begin{itemize}
\item[(i)] If $1<p<\infty$, $w$ belongs to Muckenhoupt class $A_{p}$ and $f \in L^p_w(\R^N,F^{*})$ is a solution for \eqref{eq A(D)} then $\mu$ has finite $(m,p,w)-$energy, \emph{i.e.} $\|I_{m}\mu\|_{L_{w}^{p}}<\infty$. If $w \in A_{1}$ and $f \in L^1_w(\R^N,F^{*})$ is a solution for \eqref{eq A(D)} then $\mu$ has finite $(m,1,w)-$weak energy. 
\item[(ii)] Conversely, if $|\mu|$ has finite $(m,p,w)-$energy for $1<p<\infty$ and $A(D)$ is elliptic, then there exists a function $f \in L^p_w(\R^N,F^{*})$ solving \eqref{eq A(D)}.
\end{itemize}
\end{alphatheo}
}

{The previous result extends the \cite[Theorem A]{BP} since $w(x) \equiv 1$ belongs to the Muckenhoupt classes $A_{p}$ for all $1\leq p \leq \infty$. Details about $(m,p,w)-$energy and $(m,1,w)-$weak energy are described in Subsection \ref{energy}. A refinement of the first part of Theorem \ref{theoA} for small values of $p$, in the spirit of \cite[Theorem A(i)]{BP}, for weighted powers $w(x)=|x|^{\alpha}$, is presented in Proposition \ref{prop medida nula}. }

{Our second and main result concerns the solvability of \eqref{eq A(D)} on $L^\infty_{1/w}$ in the same scope due to Moonens, Russ and Tuominen in \cite{MRT}}.

\begin{alphatheo}\label{theoB}
Let $A(D)$ be a homogeneous linear differential operator of order $1\leq m<N$ on $\R^{N}$ from $E$ to $F$, $\mu \in \M(\R^N,E^{*})$ and {$w\in L^{1}_{loc}(\R^N)$} be a {positive} weight in $\R^N$. If $A(D)$ is elliptic and canceling, and both $\nu:=|\mu|$ and $w$ satisfy the following testing conditions
	\begin{equation}\label{tcA1}
		\int_{B^{c}(0,2|y|)} \dfrac{1}{|x|^{N-m+1}} \, d\nu(x) \leq C_{1} \,\frac{w(y)}{|y|}, \quad \quad a.e. \,\, y \in \R^{N},
	\end{equation}
and
	\begin{equation}\label{tcA2}
		\int_{B(0,4|y|)} \dfrac{1}{|x-y|^{N-m}} \, d\nu(x) \leq C_{2}\, w(y),
		\quad \quad a.e. \,\,y \in \R^{N},
	\end{equation}
{for some positive constants $C_1$ and $C_2$,} then there exists $f \in L^\infty_{1/w}(\R^N,F^{*})$ solving \eqref{eq A(D)}. 
\end{alphatheo}

A natural question arises about conditions for $\nu \in \M_{+}(\R^N)$ and $w \in L^{1}_{loc}(\R^N)$ positive that satisfy \eqref{tcA1} and \eqref{tcA2}. In Example \ref{rem:100}, we present a condition on $\nu$ such that the {testing} conditions are satisfied when $w(x)=|x|^{\alpha}$ for $0<\alpha<1$. {For a special class of weights $ w $, namely $ A_1 $-weights (see Section \ref{sect prelim}), the
Proposition \ref{propmain} 
show us that the stronger conditions 
\begin{equation}\label{decorigin}
	\nu(B(0,r)) \leq {C_3}\, r^{-m} \int_{B(0,r)} w(x) \, dx, \quad \text{ for any }r>0,
\end{equation}
and
\begin{equation}\label{wolff peso}
		 \int_0^{|y|/2} \frac{\nu (B(y,r))}{r^{N-m+1}}\,dr\leq C_4\, w(y),  \quad  \text{a.e.}\,\, y \in \R^{N} , 
\end{equation}
for some positive constants $C_3$ and $C_4$, imply \eqref{tcA1} and \eqref{tcA2}. {Moreover, we observe that a}
 stronger sufficient condition for 
the potential condition \eqref{wolff peso} is given by
\begin{equation*}\label{decoutorigin}
	\nu(B(x,r)) \leq {C}|x|^{-m} \int_{B(x,r)} w(y) \, dy, \qquad \text{ for any } 0<r<\frac{|x|}{2}.
\end{equation*}

 As a direct consequence of the previous comment, we state the following extension of Theorem \ref{theorem BP}:  

\begin{alphatheo}\label{theoC}
Let $A(D)$ be a homogeneous linear differential operator of order $1\leq m<N$ on $\R^{N}$ from $E$ to $F$, $\mu \in \M(\R^N,E^{*})$, and $w$ be an  $A_1-$weight. If $A(D)$ is elliptic and canceling, and both $\nu:=|\mu|$ and $w$ satisfy \eqref{decorigin} and  \eqref{wolff peso},
then there exists $f \in L^\infty_{1/w}(\R^N,F^{*})$ solving \eqref{eq A(D)}.	
\end{alphatheo}

We remark that the condition \eqref{decorigin} for $m=1$ and $\mu \in \M_+(\R^N)$ is precisely \eqref{ahlfors MRT} restricted to balls centered at the origin. Moreover, every weight in the $A_1$ class is 1-admissible in the sense of \cite[Definition 2.2]{MRT} (see \cite[Theorem 4]{Bjorn} for more details). Therefore, Theorem \ref{theoC} extends \cite[Theorem 3.21]{MRT} to higher-order operators for a subset of their admissible weights in the sense that it weakens condition \eqref{ahlfors MRT} by \eqref{decorigin}, with the additional potential condition \eqref{wolff peso}.
\\

The paper is organized as follows. In Section \ref{sect prelim} we present definitions, basic properties and examples regarding weights, {Riesz potentials, measures with finite energy} and cocanceling operators. We prove the Theorem \ref{theoA} in Section \ref{sect theoA}. The Section \ref{sect theoB} is devoted to the proof of Theorem \ref{theoB}, including its main ingredient, {\emph{i.e.} a Stein-Weiss type inequality in weighted Lebesgue space $L^1_w$ given by Lemma \ref{main}}. {We also provide an example which recovers the previous result for power weights given in \cite{BP}. }
In the Section \ref{sect theoC}, we prove {a preliminar proposition} allowing us to obtain Theorem \ref{theoC}. We finish with {Section \ref{sect applications}, where some general comments and applications are presented.} \\

\noindent \textbf{Notations:} In this work, the symbol $f \lesssim g$ means that there exists a constant $C > 0$, independent of both $f$ and $g$ , such that $f \leq C g$. Given a {measurable} set $A \subset {\R}^N$, we denote by $|A|$ its Lebesgue measure. We write $B=B(x,r)$ for the open ball with center $x$ and radius $r>0$. 
We denote the complement of $B(x,r)$ 
by $B^c(x,r)$.
We fix $\fint_Q f(x) dx := \frac{1}{|Q|} \int_Q f(x)dx$. {Given an open subset $\Omega \subseteq \R^N$,} we denote by {$\M(\Omega)$} the set of {$\sigma$-finite} signed Borel measures on {$\Omega$}. We add the subscript {$\M_+(\Omega)$} to denote the {subset} of {non-negative} measures on {$\Omega$}. {We denote by $\M(\Omega, \mathbb{C})$ the set of complex-valued {$\sigma$-finite} Borel measures on $\Omega$ given by $\mu=\mu^{\real} + i\,\mu^{\imag}$, where $\mu^{\real}, \mu^{\imag} \in \mathcal{M}(\Omega)$, and $\M_+(\Omega, \mathbb{C})$ is the set of measures $\mu \in \M(\Omega, \mathbb{C})$ such that $\mu^{\real}, \mu^{\imag} \in \M_+(\Omega)$. For finite dimensional vector spaces $X$, we use the identification $X^* \cong X$ for simplicity.}

\section{Preliminaries}\label{sect prelim}

\subsection{Muckenhoupt weight classes }\label{subsect weights}

We say that a locally integrable function $w: \R^N \rt \R$ is a \emph{weight} if $w(x)>0$ for almost every $x \in \R^N$. {Given a measurable} set $U$, we write \[w(U) \doteq \int_U w(x) \, dx.\]}

For $1 \leq p <\infty $, {we denote by $L^p_w(\R^N,\R^M)$ the \emph{weighted $L^p$-space}}, which comprises all measurable functions $f: \R^N \rt \R^M$ such that $|f|^p w \in L^1(\R^N,\R^M)$, and we endow it with the norm
\[\|f\|_{L^p_w} := \left(\int_{\R^N} |f(x)|^p w(x) \, dx\right)^{1/p}.\]
{Similarly, we denote by $L^\infty_w(\R^N,\R^M)$ the \emph{weighted $L^\infty$-space}, which comprises all measurable functions $ f $ such that}
\[\|f\|_{L^\infty_w} := \|{f}w\|_{L^\infty}<\infty.\]

An important class of weights that will feature in this work is the class of \emph{$A_p$-weights}, defined by Muckenhoupt in \cite{Mucken}. {For $1<p<\infty$, a weight $w$ is an \emph{$A_p$-weight}, denoted by $w \in A_p$, if
\[
{[w]_{A_{p}}:=}\sup_{B} \left(\fint_B w(x) \, dx\right)\left(\fint_B w(x)^{1/(1-p)} \, dx\right)^{p-1} < \infty,
\]
and is an $A_1$-weight if}
\[
{[w]_{A_{1}}:=}\sup_{B} \left[ \left( \fint_{B} w(x) \, dx \right) \esssup\limits_{x\in B} \frac{1}{w(x)} \right]<\infty,
\]
where the supremum is taken over every ball $B \subset \R^N$. Notice that, if $w \in A_1$  then  
\begin{equation}\label{A1}
\fint_{B} w(z) \, dz \leq {[w]_{A_{1}} \essinf_{z \in B} w(z)
\leq [w]_{A_{1}} w(y)}, \qquad \text{a.e.} \quad  y \in B
\end{equation}
for every ball $B \subset \R^N$. {Denoting by $M$ the Hardy-Littlewood maximal operator given by 
$$Mf(x)=\sup_{r>0}\fint_{B(x,r)}|f|(y)dy,$$ 
then $w \in A_1$ is equivalent to the condition $Mw(x)\leq [w]_{A_{1}}w(x)$ almost everywhere $x \in \R^{N}.$.} The weight $w(x)=|x|^\alpha$ belongs to the $A_1$ class if and only if  $-N < \alpha \leq 0$ (see \cite[p. 229]{Torch}) {and to the $A_p$ class, $1<p<\infty$, if and only if $-N < \alpha < N(p-1)$ (see \cite[p. 236]{Torch}). Notice that}
\[
\fint_{B(0,r)} |x|^\alpha dx = \frac{C_N}{r^N} \int_{0}^r s^{\alpha+N-1}ds =C_{N,\alpha} \, r^\alpha
\]
{if $\alpha > -N$. When $\alpha > 0$, then 
the integral goes to infinity as $r \longrightarrow \infty$ and $\esssup\limits_{x \in B(0,r)} |x|^{-\alpha} = \infty$ for any $ r>0 $. As for $1<p<\infty$,
\[
\left(\fint_{B(0,r)} |x|^\frac{\alpha}{1-p}\, dx\right)^{p-1} = \left(\frac{C_N}{r^N} \int_{0}^r s^{\frac{\alpha}{1-p}+N-1}\, ds\right)^{p-1} = C_{N,\alpha,p} \, r^{-\alpha}
\]
if $\alpha < N(p-1)$.} We refer to \cite[Example 1.2.5]{Turesson2000} for more examples of weights in $A_p$ classes.

\subsection{Riesz potentials and measures with finite weighted energy}\label{energy}

{For $0 < m < N$ and $f$ a function in the Schwartz space $S(\R^N)$, consider the fractional integrals called \emph{Riesz potential operators} defined by
\[I_m f(x) = \dfrac{1}{\gamma(m)} \int_{\R^N} \dfrac{f(y)}{|x-y|^{N-m}} dy,\]
where $\gamma(m) := \pi^{N/2} 2^m \Gamma(m/2)/\Gamma\left({(N-m)}/{2}\right)$}. {It is clear that $\widehat{I_{m}f}(\xi)=|\xi|^{m}\widehat{f}(\xi)$ and $I_{m}$ is continuously extended from $L^{p}(\R^{N})$ to $L^{q}(\R^{N})$ if $\displaystyle{\frac{1}{q}:=\frac{1}{p}-\frac{m}{N}}$ for $1\leq p<q<\infty$ (see \cite[Chapter V]{Stein}).} For a finite dimensional complex vector space $X$ over $\C$, with $d \doteq \dim_{\C} X < \infty$, $\M(\Omega,X)$ means the set of all $X$-valued complex measures on $\Omega$, $\mu = (\mu_{1},\dots, \mu_{d})$, where $\mu_{\ell}=\mu_{\ell}^{\real} + i\,\mu_{\ell}^{\imag} \in \M(\Omega,\C)$ for all $\ell=1,\dots,d$. By $\M_+(\Omega,X)$ we mean the set of measures $\mu \in \M(\Omega,X)$ such that $\mu_{\ell} \in \M_+(\Omega,\C)$ for all $\ell=1,\dots,d$. For $\eta \in \M(\Omega,\C)$, we define the Riesz potential
\[I_m \eta(x) := \dfrac{1}{\gamma(m)} \int_{\Omega} \dfrac{1}{|x-y|^{N-m}} d\eta(y)\]
and, for measures $\mu \in \M(\Omega,X)$, we define $I_{m}\mu:=(I_{m}\mu_{1},\dots,I_{m}\mu_{d})$. {Next we present the definition of weighted energy used in Theorem \ref{theoA}.}
\begin{definition}
    Let $w$ be a weight and $\mu \in \M(\R^N,X)$. We say that $\mu$ has finite $(m,p,w)-$energy if $\|I_m \mu\|_{L^p_w} < \infty$ for $1 \leq p < \infty$ and $\mu$ has finite $(m,1,w)-$weak energy if
\[\|I_m \mu\|_{L^{1,\infty}_w} \doteq \sup_{\lambda > 0} \lambda \, w\left( \{x: |I_m \mu(x)|>\lambda\} \right) <\infty.\]
\end{definition}

We recall the definition of the \emph{Riesz transforms} $R_j$, for $j=1,\dots,N$, given by
\[R_j f(x) = \lim_{\eps \to 0^+} c_N \int_{|x-y|>\eps} f(y) \dfrac{x_j-y_j}{|x-y|^{N+1}} \, dy\]
for $f \in S(\R^N)$ and $c_N = \Gamma\left(\frac{N+1}{2}\right)/\pi^{(N+1)/2}$. {Another characterization is given by symbols $\widehat{R_j f}(\xi) = -i \dfrac{\xi_j}{|\xi|} \hat{f}(\xi)$ and clearly we conclude $\sum_{j=1}^{N}R^{2}_{j}f=-f$ for all $f \in  S(\R^N)$.}
It follows from \cite[Theorems 3.1, 3.5 and 3.7 in Chapter 4]{GCF} that the Riesz transforms are continuously extended to $L^p_w$ (\emph{i.e.} $\|R_j f\|_{L^p_w} \lesssim \|f\|_{L^p_w}$), for $1<p<\infty$, if and only if $w \in A_p$. For the limit case $p=1$, $w \in A_1$ if and only if the Riesz transforms are of weak type $(1,1)$ with respect to $w$, \emph{i.e.} for every $f \in L^1_w(\R^N)$ and $j=1,\dots,N$,
\[\|R_j f\|_{L^{1,\infty}_w} = \sup_{\lambda > 0} \lambda \, w\left( \{x: |R_j f(x)|>\lambda\} \right) \lesssim \|f\|_{L^1_w}.\]

{Now, let $\alpha=(\alpha_1,\dots,\alpha_N)$ be a multi-index. We define the \emph{Riesz transform of order $\alpha$} as the operator
\[R^\alpha f \doteq \left(R_1^{\alpha_1} \circ R_2^{\alpha_2} \circ \cdots \circ R_N^{\alpha_N}\right) f,\]
for $f \in S(\R^N)$, where $R_j^{\alpha_j}$ is the composition $R_j \circ R_j \circ \cdots \circ R_j$ for $\alpha_j$ times. Naturally, the boundedness properties of $R_j$ {on $L^{p}_{w}(\R^{N})$} are extended to $R^\alpha$.  
Next we prove the first part of Theorem \ref{theoA} as a direct consequence of the continuity of $R^{\alpha}$.}

\begin{proposition}\label{prop finite 1}
Let $A(D)$ be a homogeneous linear differential operator of order $1 \leq m<N$ on $\R^{N}$, $N\ge2$, from $E$ to $F$, $\mu \in \M(\R^N,E)$ and $w$ be a weight in $\R^N$.
{If $1<p<\infty$, $w \in A_{p}$ and $f \in L^p_w(\R^N,F^{*})$ is a solution for \eqref{eq A(D)} then $\mu$ has finite $(m,p,w)-$energy. If $w \in A_{1}$ and $f \in L^1_w(\R^N,F^{*})$ then $\mu$ has finite $(m,1,w)-$weak energy.} 
\end{proposition}
{
\begin{proof}
 The proof follows from the identity \eqref{eq:ImRalpha} obtained in  \cite[Proposition 3.1]{BP} and the boundedness of Riesz transforms of order $\alpha$ in $L^p_w(\R^{N})$ for each $|\alpha|=m$, where $w \in A_{p}$. 
\end{proof}
}

{In the case {$w(x) \equiv1$}, the Proposition 2.1 in \cite{BP} asserts that for small values of $p$, precisely for $1<p\leq N /(N-m)$, if a measure $\mu \in \M_+(\Omega,X)$ has finite $(m,p)-$energy or {finite $(m,1)-$weak energy}, then $\mu \equiv 0$ in $\Omega$. In view of Proposition \ref{prop finite 1}, it is natural {to expect} a weighted extension of this result, and one would be tempted to relate it to $A_p$ classes. The following result for power weights shows that this is not precise:}

\begin{proposition}\label{prop medida nula}
    Let $\Omega \subseteq \R^N$ be an open set, $X$ be a finite dimensional complex vector space and $\mu \in \M_+(\Omega,X)$. Let $0 < m < N$ and consider the power weight $w(x)=|x|^\alpha$ in $\R^N$. {If $1<p\leq(N+\alpha)/(N-m)$, $\alpha > -m$ and $\mu$ has finite $(m,p,w)-$energy then $\mu \equiv 0$ in $\Omega$. The same conclusion holds if $\alpha > -N$ and $\mu$ has finite $(m,1,w)-$weak energy.}
\end{proposition}

\begin{proof} 
{The proof is \textit{bis in idem} of Proposition 2.1 in \cite{BP} and it will be present for the sake of completeness.} Without loss of generality, assume $\mu_\ell$ are {non-negative} measures for each $\ell$. For each $r>0$, we have
	\begin{equation}\label{eq:Immuell}
	I_m \mu_{\ell}(x)
    \gtrsim \int_{\Omega \cap B(0,r)} \dfrac{1}{|x-y|^{N-m}} \, d\mu_{\ell}(y)
        \geq \dfrac{\mu_{\ell}(\Omega \cap B(0,r))}{(|x|+r)^{N-m}}.
	\end{equation}
    Hence, for $1<p<\infty$,
    \begin{align*}
    \int_{\R^N} |I_m \mu(x)|^p \, |x|^\alpha \, dx &\gtrsim \int_{\R^N} [I_m \mu_{\ell}(x)]^p \, |x|^\alpha \, dx \gtrsim \int_{\R^N} \left[\dfrac{\mu_{\ell}(\Omega \cap B(0,r))}{(|x|+r)^{N-m}}\right]^p \, |x|^\alpha \, dx\\
    &= \left[\mu_{\ell}(\Omega \cap B(0,r))\right]^p \int_{\R^N} \dfrac{|x|^\alpha}{(|x|+r)^{(N-m)p}} \, dx.
    \end{align*}
    However 
    \[\int_{\R^N} \dfrac{|x|^\alpha}{(|x|+r)^{(N-m)p}} \, dx = |S^{N-1}| \int_0^\infty \dfrac{s^{\alpha+N-1}}{(s+r)^{(N-m)p}} \, ds\]
    and this last integral blows up to infinity when $1<p \leq \frac{N+\alpha}{N-m}$. Therefore, in this case, since $\mu$ has finite $(m,p,w)-$energy, we conclude that $\mu_{\ell}(\Omega \cap B(0,r))=0$. \\
    For $p=1$, {from \eqref{eq:Immuell} we have} 
    \[\sup_{\lambda > 0} \lambda \, w\left(\left\{x \in \R^{N}\,:\, \frac{\mu_{\ell}(\Omega \cap B(0,r))}{(|x|+r)^{N-m}} > \lambda\right\}\right) \lesssim  
    \|I_m \mu\|_{L^{1,\infty}_w}<\infty.\]
   Thus,
    \begin{align*}
	\lambda\,w\left(\left\{x: \frac{\mu_{\ell}(\Omega \cap B(0,r))}{(|x|+r)^{N-m}}>\lambda\right\}\right)
	&= \lambda \int_{B\left(0,\left(\frac{\mu_{\ell}(\Omega \cap B(0,r))}{\lambda}\right)^{1/(N-m)}-r\right)} |x|^\alpha \, dx \\
	&= \dfrac{|S^{N-1}|}{\alpha+N}\lambda^{-\frac{m+\alpha}{N-m}}\, \left(\left[\mu_{\ell}\left(\Omega \cap B(0,r)\right)\right]^{\frac{1}{N-m}}-\lambda^{\frac{1}{N-m}}r\right)^{\alpha+N},
\end{align*}
    which blows up to infinity when $\lambda > 0$ is small and $\mu_{\ell}\left(\Omega \cap B(0,r)\right) \neq 0$. Hence, once again, $\mu_{\ell}\left(\Omega \cap B(0,r)\right)=0$. Since $r>0$ was chosen arbitrarily, and $\ds{\Omega = \bigcup_{k \in \N} [\Omega \cap B(0,k)]}$, we must have $\mu_\ell \equiv 0$ on $\Omega$ for every $\ell$. Therefore, $\mu \equiv 0$.
\end{proof}

{
We remark that, from the properties of $A_p$ weights discussed in Subsection \ref{subsect weights}, the following situations occur: 
if $N(p-1)-mp \leq \alpha < N(p-1)$, then $|x|^\alpha \in A_p$ and Proposition \ref{prop medida nula} holds; if $\alpha \geq N(p-1)$, then $|x|^\alpha \notin A_p$ but Proposition \ref{prop medida nula} still holds; if $-N < \alpha < N(p-1)-mp$, then $|x|^\alpha \in A_p$ but the Proposition \ref{prop medida nula} can not be applied to conclude $ \mu \equiv 0 $. The same argument holds for $ p=1 $: if $-N < \alpha \leq 0$, then $|x|^\alpha \in A_1$ and Proposition \ref{prop medida nula} holds; if $\alpha > 0$, then $|x|^\alpha \notin A_1$ but Proposition \ref{prop medida nula} still holds.
}

\subsection{Canceling and cocanceling operators}\label{subsect cocanc}

{Van Schaftingen proved the following important property of elliptic operators in \cite[Proposition 4.2]{VS}:} {if $A(D): \cinf(\R^N,E) \rt \cinf(\R^N,F)$ is an elliptic homogeneous differential operator, then there exists another homogeneous linear differential operator $L(D): \cinf(\R^N,F) \rt \cinf(\R^N,V)$, for some finite dimensional complex vector space $V$, such that
	$$ker\,L(\xi) = A(\xi)[E]$$
	for every $\xi \in \R^N \setminus \{0\}$.}

A homogeneous linear differential operator $L(D)$  
on $\R^{N}$ from $F$ to $V$ is called 
\emph{cocanceling} if 
$$\displaystyle{\bigcap_{\xi \in \R^{N}\backslash \left\{ 0 \right\}}ker\,L(\xi)=\left\{ 0 \right\}}.$$
{Consequently, if $A(D)$ is elliptic and canceling, then the operator $L(D)$ from {\cite[Proposition 4.2]{VS}}
is cocanceling{, and thus $L(D)(A(D)u) = 0$} for every ${u \in \cinf(\R^N,E)}$.
}

A classic example of a cocanceling operator on $\R^N$ from $F=\R^N$ to $V=\R$ is the divergence operator $L(D)=\diver$. {In fact, $L(\xi)[e]=\xi \cdot e$ for every $e \in \R^N$, and therefore $ ker\,L(\xi)= \xi^{\perp}$}. {We refer the \cite[Section 3]{VS} for more examples of cocanceling operators.}  
The following estimate for vector fields belonging to the kernel of some cocanceling operator was presented {in} \cite{DNP}:

\begin{lemma}[{\cite[Lemma 3.1]{DNP}}]\label{mainlemma}
Let $L(D)$ be a {cocanceling} homogeneous linear differential operator of order m on $\R^{N}$ from $F$ to $V$. {Then} 
there exists $C>0$ 
such that for every $\varphi \in C_{c}^{m}(\R^N,F)$,  
we have 
\begin{equation*}
\left| \int_{\R^N} \varphi(y)\cdot f(y)\,dy\right|\le 
C\sum_{j=1}^m \int_{\R^N}|f(y)|\,|y|^j\, |D^j\varphi (y)|\,dy.
\end{equation*}
for all functions $f \in L^{1}(\R^N,F)$ satisfying $L(D)f=0$  in the sense of distributions.
\end{lemma}

We also state {the following} self-improvement of a special Hardy-type {inequality,} first proved in \cite[Lemma 2.1]{DNP}.

\begin{lemma}\label{weight1} 
Let $1 \leq q <\infty$ and {$\nu \in \M_+(\R^N)$ be a $\sigma$-finite measure.} Suppose $u$ and $v$ are measurable, {non-negative} and finite almost everywhere functions. Then 
\begin{equation*}
\left[ \int_{\R^{N}} \left(\int_{B(0,|x|/2)}g(y)dy \right)^{q}{u}(x)d\nu(x) \right]^{1/q} \leq C \int_{\R^N} g(x){v}(x)dx
\end{equation*}
holds {for all $g \geq 0$}  if 
\begin{equation*}
\left( \int_{B^c(0,2|y|)} u(x) \, d\nu(x) \right)^{1/q} [v(y)]^{-1}\leq C, \quad a.e. \,\,  y \in \R^{n}.
\end{equation*}
\end{lemma}

The proof is exactly the same found at {\cite[Lemma 4.1]{BP}} and will be omitted.

\section{Proof of Theorem \ref{theoA}}\label{sect theoA}

The first part of Theorem \ref{theoA} follows from Proposition \ref{prop finite 1}. To prove the second part, for $1 < p < \infty$ and a weight $w$, let $w' \doteq w^{1/(1-p)}$ be the \emph{conjungate weight} of $w$ with respect to $p$. Let $W^{m,p'}_{A,w'}(\R^N,E)$ be the closure of $\ccinf(\R^N,E)$ with respect to the norm $\|\varphi\|_{m,p',w'} \doteq \|A(D)\varphi\|_{L_{w'}^{p'}}$, where $p'=p/(p-1)$ denotes the conjugate exponent of $p$.  
{We claim that} there exists $C>0$ such that 
\begin{equation}\label{maineq p}
    \left|\int_{\R^N} \varphi(x) \, d\mu(x) \right|  \leq C \|A(D)\varphi\|_{L_{w'}^{p'}}, \quad  {\text{for all}} \,\, \varphi \in \ccinf(\R^N,E).
\end{equation}
{Assuming the validity of the previous inequality}, it implies $\mu \in [W^{m,p'}_{A,w'}(\R^N,E)]^*$. Since $A(D): W^{m,p'}_{A,w'}(\R^N,E) \to L_{w'}^{p'}(\R^N,F)$ is a linear isometry, its adjoint $A^*(D): L_{w}^{p}(\R^N,F) \to [W^{m,p'}_{A,w'}(\R^n,E)]^*$ is surjective. Therefore, there exists $f \in L_{w}^{p}(\R^N,F)$ such that $A^*(D)f = \mu$. 
Hence, it is sufficient to prove \eqref{maineq p} as we do in the following:

\begin{proposition}\label{prop finite 2}
    Let $A(D)$ be an elliptic homogeneous linear differential operator of order $1 \leq m<N$ on $\R^{N}$, $N\ge2$, from $E$ to $F$, $\mu \in \M(\R^N,E)$ and {$w\in L^{1}_{loc}(\R^N)$} be a weight in $\R^N$. If $|\mu|$ has finite $(m,p,w)-$energy for {$1<p<\infty$}, then there exists a positive constant $C$ such that \eqref{maineq p} holds.
\end{proposition}

\begin{proof}
    Since $A(D)$ is elliptic, consider 
 the function $\xi\mapsto H(\xi)\in \mathcal{L}(F,E)$ defined by
$H(\xi)=(A^{\ast}\circ A)^{-1}(\xi)A^{\ast}(\xi)$
that is smooth in $\R^{N}\backslash \left\{0\right\} $ and homogeneous of 
degree $-m$, where $A^{\ast}(\xi)$ is the symbol of the adjoint operator $A^{\ast}(D)$.
Assuming $1 \leq m<N$, we have that $H$ is a locally integrable tempered {distribution, and} its inverse Fourier transform $K(x)$
is a locally integrable tempered distribution homogeneous of degree $-N+m$ (see \cite[p.~71]{Javier}). That implies the identity 
\begin{equation}\label{decfrac}
\varphi(x)=\int_{\R^{N}}K(x-y)[A(D)\varphi(y)]\,dy, \quad {\text{for all}} \,\,\varphi \in C_{c}^{\infty}(\R^{N},E),
\end{equation}
and $|\varphi(x)| \lesssim I_m|A(D)\varphi|(x) = C \ds\int_{\R^N}\dfrac{|A(D)\varphi(y)|}{|x-y|^{N-m}}\,dy$. Hence, {applying H\"older's inequality we have}
	\begin{align*}
		\left|\int_{\R^N} \varphi(x) \, d\mu(x) \right| &\lesssim \int_{\R^N} \left[\int_{\R^N} \dfrac{|A(D)\varphi(y)|}{|x-y|^{N-m}} \, dy\right] d|\mu|(x)\\
		&\lesssim  \int_{\R^N} |A(D)\varphi(y)| \, I_m |\mu|(y) \, dy\\
		& \leq \|A(D)\varphi\|_{L^{p'}_{w'}} \, \|I_m |\mu|\|_{L^p_w}\\
        &\lesssim \|A(D)\varphi\|_{L^{p'}_{w'}},
	\end{align*}
	since $|\mu|$ has finite $(m,p,w)-$energy, {and the inequality \eqref{maineq p} follows as desired.}
\end{proof}

\section{Proof of Theorem \ref{theoB}}\label{sect theoB}

In order to prove Theorem \ref{theoB}, it is sufficient to exhibit $C>0$ such that 
	\begin{equation}\label{maineq}
		\left|\int_{\R^N} \varphi(x) \, d\mu(x) \right|  \leq C \|A(D)\varphi\|_{L_{w}^{1}}, \quad  {\text{for all}} \,\, \varphi \in \ccinf(\R^N,E),
	\end{equation}
which implies $\mu \in [W^{m,1}_{A,w}(\R^N,E)]^*$. 
{Following the same steps from the argument of the previous proposition, we conclude} $A^*(D): L_{1/w}^{\infty}(\R^N,F) \to [W^{m,1}_{A,w}(\R^n,E)]^*$ is surjective  ({note that, for $p=1$, the conjugate weight of $w$ is given by $w':=1/w$).} Therefore, there exists a solution $f \in L_{1/w}^{\infty}(\R^N,F)$ for $A^*(D)f = \mu$.

Now, we move on to prove the inequality \eqref{maineq}.
Since $A(D)$ is elliptic, 
{we have identity \eqref{decfrac}} with
\begin{equation} \label{ka} 
 |K(x-y)| \leq C \; |x - y|^{m-N} , \quad \,  {x} \neq y,
\end{equation}
and 
\begin{equation} \label{kb0} 
 |\partial_y K(x-y)| \leq C \; |x - y|^{m-N-1} , \quad \,  2|y| \leq |x|.
\end{equation}
Applying the identity \eqref{decfrac} into \eqref{maineq}, it is sufficient to obtain  
\begin{equation}\label{eqmain}
	\int_{\R^N}  \left|\int_{\R^N}K(x-y)g(y) \, dy\right| \, d|\mu|(x)  \lesssim \int_{\R^N}|g(x)|w(x)dx
\end{equation}
for all $g:=A(D)\varphi$, with $\varphi \in \ccinf(\R^N,E)$. Integral estimates like \eqref{eqmain} in the general setting 
\begin{equation*}
\left( \int_{\R^N} \left| \int_{\R^N} K(x,y)g(y) \; dy \right|^q d\nu(x)  \right)^{1/q} 
\leq C  \int_{\R^N} |g(x)| w(x) \; dx
\end{equation*}
were studied by the third author \emph{et al.} in \cite{DNP,HHP} for $d\nu:=|x|^{-\beta q} \;dx$ and $w(x)=|x|^{\alpha}$ (satisfying a special relation between $\alpha,\beta$ and $q$), and by the first and third authors in \cite{BP} for general measures $\nu \in \mathcal{M}_{+}(\R^{N})$ satisfying \eqref{ahlfors BP1} and \eqref{wolff BP1}, where $g$ belongs to the kernel of some cocanceling operator. Next we present a  self-improvement of \cite[Lemma 4.3]{BP} that is a  Stein-Weiss type inequality in $ L^1_w$-norm for this special class of vector fields:

\begin{lemma}\label{main}
Assume $N \geq 2$, $0 < \ell < N$ and $K(x,y) \in L^{1}_{loc}(\R^N \times \R^N, \mathcal{L}(F,V))$ 
satisfying
\begin{equation} \label{kc} 
	|K(x, y)| \leq {C_1} |x - y|^{\ell-N} , \quad x \neq y ,
\end{equation}
and 
\begin{equation} \label{kd}
	|K(x,y)-K(x,0)| \leq {C_2} \frac{|y|}{|x|^{N-\ell+1}} , \quad 2|y|\leq |x|.
\end{equation}
Suppose $1 \leq q < \infty$, and let $w$ be a weight and $\nu \in \mathcal{M}_{+}(\R^{N})$ satisfy the testing conditions
\begin{equation} \label{12a}
	{\left(\int_{B^{c}(0,2|y|)} \dfrac{1}{|x|^{(N-\ell+1)q}} \, d\nu(x)\right)^{1/q} \leq C_{3} \, \frac{w(y)}{|y|}}, \quad \quad a.e. \,\, y \in \R^{N},
\end{equation}
{and}
\begin{equation} \label{13a}
	\left(\int_{B(0,4|y|)} \dfrac{1}{|x-y|^{(N-\ell)q}} \, d\nu(x)\right)^{1/q} \leq C_{4}\, w(y),
	\quad \quad a.e. \,\,y \in \R^{N}.
\end{equation}
If $L(D)$ is cocanceling, then there exists $C>0$ such that 
\begin{equation*}\label{mainineq}
	\left( \int_{\R^N} \left| \int_{\R^N} K(x,y)g(y) \; dy \right|^q d\nu(x)  \right)^{1/q} 
\leq C \int_{\R^{N}} |g(x)|w(x) \, dx, 
\end{equation*}
for all $g \in {L^1_w}(\R^N,F)$ satisfying $L(D)g=0$  
in the sense of distributions. 
\end{lemma}

In order to prove inequality \eqref{eqmain}, and consequently Theorem \ref{theoB}, we estimate 
\begin{equation}\nonumber
	\left|\int_{\R^N}  \left[\int_{\R^N}K(x-y)g(y) \, dy\right] \, d\mu(x) \right| \leq
	\int_{\R^N}  \left|\int_{\R^N}K(x-y)g(y) \, dy \right| \, d|\mu|(x) 
\end{equation}
and apply the Lemma \ref{main} for {$q=1$, $\ell=m$ and $\nu=|\mu|$,} taking $K(x,y)=K(x-y)$ given by identity \eqref{decfrac} that satisfies  \eqref{ka} {and \eqref{kb0}}. Clearly \eqref{kb0} implies   \eqref{kd} from the mean value inequality. 
The conclusion follows by taking $g:=A(D)u$, {which} belongs to the kernel of some cocanceling operator, since $A(D)$ is {elliptic and} canceling.

\begin{proof}[Proof of Lemma \ref{main}]
Let $\psi \in C_c^\infty(B(0,1/2))$ be a cut-off function such that $0 \leq \psi \leq 1$, $\psi \equiv 1$ on $B(0,1/4)$, and write $K(x,y) = K_1(x,y)+K_2(x,y)$ with $K_1(x,y) = \psi(y/|x|)K(x,0)$. It is enough to show {that for $ j=1,2 $}, we have
	\begin{equation*}\label{Ji1}
		J_j \doteq \left(\int_{\R^N} \left| \int_{\R^N} K_j(x,y) f(y) \, dy \right|^q \, d\nu(x) \right)^{1/q} \leq C \int_{\R^N} |f(x)| w(x) \, dx.
	\end{equation*}
Using the control \eqref{kc} and Lemma \ref{mainlemma}, we may estimate
	\begin{align*}
		J_1 &= \left(\int_{\R^N} \left| \int_{\R^N} \psi\left(\dfrac{y}{|x|}\right) f(y) \, dy \right|^q |K(x,0)|^q \,  d\nu(x) \right)^{1/q}\\
		&\lesssim \left(\int_{\R^N} \left| \int_{\R^N} \psi\left(\dfrac{y}{|x|}\right) f(y) \, dy \right|^q |x|^{(\ell - N)q} \,  d\nu(x) \right)^{1/q} \\
		&\lesssim \left(\int_{\R^N} \left( \int_{B(0,|x|/2)} \dfrac{|y|}{|x|} |f(y)| \, dy \right)^q |x|^{(\ell - N)q} \,  d\nu(x) \right)^{1/q} \\
		&= \left(\int_{\R^N} \left( \int_{B(0,|x|/2)} |y| |f(y)| \, dy \right)^q |x|^{(\ell - N-1)q} \,  d\nu(x) \right)^{1/q}.
	\end{align*}
{From assumption \eqref{12a}, we can apply the Hardy type inequality given by Lemma \ref{weight1} with $u(x):=|x|^{(\ell - N-1)q}$, $ v(y):={w(y)}|y|^{-1}$ and $ g(y):=|y||f(y)| $, yielding}
	\begin{align*}
		J_1 & \lesssim \left(\int_{\R^N} \left( \int_{B(0,|x|/2)} |y| |f(y)| \, dy \right)^q |x|^{(\ell -N-1)q}  \, d\nu(x) \right)^{1/q}  \lesssim  \int_{\R^N} |f(x)| w(x) \, dx. 
	\end{align*}	
		On the other hand, note that to estimate $ J_2 $ it is enough to show
	 \begin{equation*}\label{eq:3002}
	 	\left( \int_{\R^N} |K_2(x,y)|^q \, d\nu(x) \right)^{1/q} \leq C\, w(y), \quad \quad a.e. \, \, y \in \R^{n} ,	
	 \end{equation*}
	since, by using Minkowski's inequality, we obtain
	\begin{align*}
		J_2 &\leq \int_{\R^N} \left( \int_{\R^N} |K_2(x,y)|^q  \,{d\nu(x)} \right)^{1/q} |f(y)| \, dy \lesssim  \int_{\R^N} |f(y)| w(y) \, dy.
	\end{align*}
Writing 
	$
	K_2(x,y)=\psi \left( \frac{y}{|x|} \right) \left[ K(x,y)-K(x,0) \right] + \left[ 1-\psi\left( \frac{y}{|x|} \right) \right] K(x,y),
	$
we may control 
	$
	|K_2(x,y)| = |K(x,y) - K(x,0)|\lesssim \frac{|y|}{|x|^{N-\ell+1}}
	$
for $4|y| \leq |x| $,
and for $2|y|> |x|$ we have
	$
	|K_2(x,y)|=| K(x,y)| \lesssim \frac{1}{| x-y |^{N-\ell}}
	$.
For the {complementary part},  since $ 0 \leq  \psi \leq 1 $, we have 		
	\begin{align*}
	|K_2(x,y)| \leq \left| K(x,y) \right| + \left| K(x,y)-K(x,0) \right| \lesssim \frac{1}{|x-y|^{N-\ell}}  +\frac{|y|}{|x|^{N-\ell+1}}.
	\end{align*}
	{In this way, we can write}
	\begin{equation*}
	 \int_{\R^N} |K_2(x,y)|^q \, d\nu(x)
	\lesssim \int_{|x| < 4|y| } \frac{1}{|x-y|^{(N-\ell)q}} \, d\nu(x)  +
	 \int_{2|y|\leq |x|} \frac{|y|^q}{|x|^{(N-\ell+1)q}} \, d\nu(x)
	\end{equation*}
and 
the conclusion follows from \eqref{12a} and \eqref{13a}. 
\end{proof}

\begin{example}\label{rem:100}
	If $ w(y)=|y|^\alpha $, with $0<\alpha<1$, and $\nu \in \mathcal{M}_{+}(\R^{N})$ satisfying
	\begin{equation}\label{eq:decayn}
	\nu(B(0,r)) \leq {C} r^{(N-\ell+\alpha)q} \quad \text{for any }\, r>0,
	\end{equation}
then \eqref{12a} and \eqref{13a} are satisfied. 
In effect, from decay \eqref{eq:decayn} for balls centered at the origin, we proceed as {follows:}
	\begin{align*}
		\int_{B^c(0,2|y|)}\frac{1}{|x|^{(N-\ell+1)q}}\,  d\nu(x) &=\sum_{k=1}^{\infty} \int_{2^{k}|y|\leq |x| < 2^{k+1}|y|} \frac{1}{|x|^{(N-\ell+1)q}}\, d\nu(x) \\ 
		& \leq \sum_{k=1}^{\infty} (2^k |y|)^{(\ell-N-1)q}\,  \nu\left(B(0, 2^{k+1}|y|)\right) \\
		& \lesssim \sum_{k=1}^{\infty} (2^k |y|)^{(\ell-N-1)q}\, (2^{k+1}|y|)^{(N-\ell+\alpha) q} \\
		& = |y|^{(\alpha-1)q}\, 2^{(N-\ell+\alpha)q} \sum_{k=1}^\infty 2^{(\alpha-1)kq} 
		 = C_{N,\ell,q} \left[ \frac{|y|^{\alpha}}{|y|} \right]^{q}.
	\end{align*}
In order to show \eqref{13a}, we {decompose} 
\begin{align}
	\int_{B(0,4|y|)} \frac{1}{|x-y|^{(N-\ell)q}}\, d\nu(x) &= \int_{B\left(y,\frac{|y|}2\right)} \frac{1}{|x-y|^{(N-\ell)q}} \,d\nu(x)  \nonumber\\
		&\quad  + \int_{B(0,4|y|) \setminus B\left(y,\frac{|y|}2\right)} \frac{1}{|x-y|^{(N-\ell)q}} \, d\nu(x) \nonumber\\
	 &:= {(I) + (II)}. \label{partAB}
	\end{align}
%
{For the second term we have} 
	\begin{align*}
		{(II)}\leq \left[\frac{|y|}{2}\right]^{-(N-\ell)q} \nu\left(B(0,4|y|)\right) \lesssim \left[\frac{|y|}{2}\right]^{-(N-\ell)q} (4|y|)^{(N-\ell+\alpha)q} \lesssim |y|^{\alpha q}. 
	\end{align*}
	On the other hand
	\begin{align*}
	(I) =\sum_{k=1}^\infty \int_{2^{-(k+1)}|y|\leq |z|<2^{-k}|y|} |z|^{(\ell-N)q}\, d\nu(z) 
	& \lesssim 2^{(N-\ell)q}\, |y|^{(\ell-N)q} \sum_{k=1}^\infty 2^{(N-\ell)qk}\, (2^{-k}|y|)^{(N-\ell+\alpha)q} \\
	& = 2^{(N-\ell)q}\, |y|^{\alpha q} \sum_{k=1}^\infty 2^{-\alpha q k}
	= C_{N,\ell,q,\alpha} \ |y|^{\alpha q}.
	\end{align*}
\end{example}

\begin{remark}
Combining the previous example and Lemma \ref{main} we recover the {Lemma 4.4} in \cite{BP}. 
\end{remark}

 	\begin{corollary}
 		Let $A(D)$ be a homogeneous linear differential operator of order $1\leq m<N$ on $\R^{N}$ from $E$ to $F$ and $ 0<\alpha<1 $. If $A(D)$ is elliptic and canceling, and $\mu \in \M(\R^N,E^{*})$  satisfies
 		\begin{equation*}
 			|\mu|(B(0,r)) \leq C\, r^{N-m+\alpha} , \quad \text{ for any } r>0,
 		\end{equation*}
		then there exists $f \in L^\infty_{1/w}(\R^N,F^{*})$, with $w(x)=|x|^{\alpha}$, solving \eqref{eq A(D)}.	
 	\end{corollary}

\section{Proof of Theorem \ref{theoC}}\label{sect theoC}

As a direct consequence of Theorem \ref{theoB}, it is sufficient to prove that, if $\mu$ is a Borel (complex) measure taking values in $E^{*}$ and $w$ is a weight in the $A_1$ class satisfying \eqref{decorigin} and \eqref{wolff peso}, then the testing conditions \eqref{tcA1} and \eqref{tcA2} are satisfied for $\nu:=|\mu|$. It follows from the next slight generalization of stronger conditions for the case $q=1$. 
\begin{proposition}\label{propmain}
	Suppose that $1 \leq q < \infty$, $0 < \ell < N$, $\nu \in \mathcal{M}_{+}(\R^{N})$, and  {$w \in L^{1}_{loc}(\R^{N})$ positive}. If there exists $C>0$ such that 
	\begin{equation}\label{decorigin2}
	\nu(B(0,r)) \leq {C}\left( r^{-\ell} \int_{B(0,r)} w(y)\, dy \right)^q 
	\end{equation}
holds for any $r>0$, then 
\begin{equation} \label{12aa}
	{\left(\int_{B^{c}(0,2|y|)} \dfrac{1}{|x|^{(N-\ell+1)q}} \, d\nu(x)\right)^{1/q} \leq C_{3} \, \frac{{Mw(y)}}{|y|}}, \quad \quad a.e. \,\, y \in \R^{N},
\end{equation}
Moreover, if 
	\begin{equation}\label{decoutorigin2}
	\nu(B(x,r)) \leq {C}\left(|x|^{-\ell} \int_{B(x,r)} {w(y)\, dy} \right)^q 
		\end{equation}
holds for any $r<\frac{|x|}{2}$, then 
\begin{equation} \label{13aa}
	\left(\int_{B(0,4|y|)} \dfrac{1}{|x-y|^{(N-\ell)q}} \, d\nu(x)\right)^{1/q} \leq C_{4}\, {Mw(y)},
	\quad \quad a.e. \,\,y \in \R^{N}.
\end{equation}
{Moreover, if $\nu$ satisfy \eqref{decorigin2} and \eqref{decoutorigin2} with $w$ belongs to the class  $A_{1}$ then \eqref{12a}
and \eqref{13a} are fullfilled.}   
\end{proposition}

\begin{proof}
Following the same arguments to obtain \eqref{12a} in Example \ref{rem:100}, we have
	\begin{align*}
	\int_{B^c(0,2|y|)} \frac{1}{|x|^{(N-\ell+1)q}}d\nu(x) 
	& \leq \sum_{k=1}^{\infty} (2^k |y|)^{(\ell-N-1)q}\, \nu\left(B(0, 2^{k+1}|y|)\right) \\
	& {\lesssim} \sum_{k=1}^{\infty} (2^k |y|)^{(\ell-N-1)q}\, (2^{k+1}|y|)^{(-\ell+N)q} \left(\fint_{B(0, 2^{k+1}|y|)} w(x)\, dx \right)^q \\
	& {\lesssim} \sum_{k=1}^{\infty} (2^k |y|)^{(\ell-N-1)q}\, (2^{k+1}|y|)^{(-\ell+N)q} \left(\fint_{\textcolor{blue}{B(y, 2^{k+2}|y|)}} w(x)\, dx \right)^q \\
	& \lesssim  \left[ \frac{{Mw(y)}}{|y|} \right]^q \sum_{k=1}^{\infty} 2^{-qk} \\
 &\lesssim \left[ \frac{{Mw(y)}}{|y|} \right]^q,
	\end{align*}
as desired.

Splitting  \eqref{13aa}  analogously as in \eqref{partAB}, we may control 
	\begin{align*}
	{(II)}\leq \left( \frac{|y|}{2} \right)^{(\ell-N)q} \nu(B(0,4|y|))
	&\lesssim |y|^{-Nq}\, |B(0,4|y|)|^q \left( \fint_{B(0,4|y|)} w(x)\, dx\right)^q \\
	&\lesssim |y|^{-Nq}\, |B(0,4|y|)|^q {\left( \fint_{B(y,5|y|)} w(x)\, dx\right)^q }\\
	&\lesssim  {[Mw(y)}]^q.
	\end{align*}
Now we estimate $ {(I)} $. Setting $ S_x := \left\{ r\in \R : \ r>|x-y| \right\}$ and since {$ B\left(y,|y|/2 \right)\subset B(0,2|y|) $}, we have
	\begin{align*}
	{(I)}&= (N-\ell)q \int_{\R^N} \chi_{B\left(y, \frac{|y|}{2}\right)}(x) \left( \int_{|x-y|}^\infty r^{(\ell-N)q-1}\, dr \right) d\nu(x) \\
	&= (N-\ell)q \int_{\R^N} \chi_{B\left(y, \frac{|y|}{2}\right)}(x) \left( \int_{0}^\infty \frac{\chi_{S_x}(r)}{r^{(N-\ell)q+1}}\, dr \right) d\nu(x) \\
	&= (N-\ell)q  \int_{0}^\infty \left(\int_{ B\left(y,\frac{|y|}{2}\right)\cap B(y,r)} \frac{1}{r^{(N-\ell)q+1}}\, d\nu(x)\right) dr  \\
	& = (N-\ell)q  \left( \int_{0}^{|y|/2} \frac{ \nu\left(B(y,r)\right)}{r^{(N-\ell)q+1}}\, dr +
	\int_{|y|/2}^\infty \frac{\nu\left(B(y,\frac{|y|}{2})\right)}{r^{(N-\ell)q+1}} \, dr \right). 
	\end{align*}
The second term is {controlled} exactly as $(II)$ above (using only \eqref{decorigin2}), since
\begin{align*}
\int_{|y|/2}^\infty \frac{ \nu\left(B(y,\frac{|y|}{2})\right)}{r^{(N-\ell)q+1}} \, dr 
& \lesssim  |y|^{(\ell-N)q}\, |y|^{-\ell q}\, |B(0,2|y|)|^q \left( \fint_{B(0,2|y|)} w(x)\, dx \right)^q \\
& \lesssim  |y|^{(\ell-N)q}\, |y|^{-\ell q}\, |B(0,2|y|)|^q {\left( \fint_{B(y,3|y|)} w(x)\, dx \right)^q} \\ 
 &\lesssim {{[Mw(y)]^q}.}
 \end{align*}
For the remaining term, we use \eqref{decoutorigin2} to obtain
	\begin{align*}
	\int_{0}^{|y|/2}  \frac{\nu (B(y,r))}{r^{(N-\ell)q+1}}\, dr &\lesssim |y|^{-\ell q} \, \int_{0}^{|y|/2} r^{(\ell-N)q-1} \, |B(y,r)|^q \left(\fint_{B(y,r)} w(x)dx \right)^q dr\\
	& \lesssim {[Mw(y)}]^q\ |y|^{-\ell q}  \int_{0}^{|y|/2} r^{\ell q-1} dr \\
	& \lesssim \, {{[Mw(y)]^q}}.
	\end{align*}
Combining the previous estimates we obtain  \eqref{13aa}.
 \end{proof}


{
As an immediate consequence of the previous result and  the proof of Lemma \ref{main},
we obtain the following self-improvement.
\begin{lemma}
Assume $N \geq 2$, $0 < \ell < N$ and $K(x,y) \in L^{1}_{loc}(\R^N \times \R^N, \mathcal{L}(F,V))$ 
satisfying  \eqref{kc} and \eqref{kd}. 
Suppose $1 \leq q < \infty$, and let $w$ be a weight and $\nu \in \mathcal{M}_{+}(\R^{N})$ satisfy 
 \eqref{decorigin2} and \eqref{decoutorigin2}. 
If $L(D)$ is cocanceling, then there exists $C>0$ such that 
\begin{equation*}\label{mainineq}
	\left( \int_{\R^N} \left| \int_{\R^N} K(x,y)g(y) \; dy \right|^q d\nu(x)  \right)^{1/q} 
\leq C \int_{\R^{N}} |g(x)|Mw(x) \, dx, 
\end{equation*}
for all $g \in {L^1_{Mw}}(\R^N,F)$ satisfying $L(D)g=0$ in the sense of distributions. 
\end{lemma}}

\section{Applications and general comments}\label{sect applications}

\subsection{Non-weighted Lebesgue solvability} The Theorem \ref{theoC} applied to the case when $w\equiv 1$, which belongs to the class of $A_1$-weights, recovers precisely the Theorem \ref{theorem BP}. The following stronger version is derived from Theorem \ref{theoB}:

\begin{theorem}
Let $A(D)$ be a homogeneous linear differential operator of order $1\leq m<N$ on $\R^{N}$ from $E$ to $F$ and $\mu \in \M(\R^N,E^{*})$. If $A(D)$ is elliptic and canceling, and $\nu:=|\mu|$  satisfies 
	\begin{equation*}\label{tcA1a}
		\int_{B^{c}(0,2|y|)} \dfrac{1}{|x|^{N-m+1}} \, d\nu(x) \leq C_{1}\, \frac{1}{|y|}, \quad \quad a.e. \,\, y \in \R^{N},
	\end{equation*}
and
	\begin{equation*}\label{tcA2a}
		\int_{B(0,4|y|)} \dfrac{1}{|x-y|^{N-m}} \, d\nu(x) \leq C_{2},
		\quad \quad a.e. \,\,y \in \R^{N},
	\end{equation*}
{for some positive constants $C_1$ and $C_2$,} then there exists $f \in L^\infty(\R^N,F^{*})$ solving \eqref{eq A(D)}. 
\end{theorem}


\subsection{Two-weighted inequalities} Let $u$ and $v$ be measurable, {non-negative} and finite almost everywhere {functions}. A particular case of Lemma \ref{main} can be stated as follows:

\begin{lemma}
Assume $N \geq 2$, $1 \leq q <\infty$,  $0 < \ell < N$ and $K(x,y) \in L^{1}_{loc}(\R^N \times \R^N, \mathcal{L}(F,V))$ 
satisfying \eqref{kc} and \eqref{kd}. Suppose $u$ and $v$ satisfy the testing conditions
\begin{equation} \label{12a1}
	{\left(\int_{B^{c}(0,2|y|)} \dfrac{u(x)}{|x|^{(N-\ell+1)q}} \, dx\right)^{1/q} \leq C_{3}\, \frac{v(y)}{|y|}}, \quad \quad a.e. \,\, y \in \R^{N},
\end{equation}
{and}
\begin{equation} \label{13a1}
	\left(\int_{B(0,4|y|)} \dfrac{u(x)}{|x-y|^{(N-\ell)q}} \, dx\right)^{1/q} \leq C_{4} \, v(y),
	\quad \quad a.e. \,\,y \in \R^{N}.
\end{equation}
If $L(D)$ is cocanceling, then there exists $C>0$ such that 
\begin{equation}\label{mainineq-two-weight}
	\left( \int_{\R^N} \left| \int_{\R^N} K(x,y)g(y) \; dy \right|^q u(x)\, dx  \right)^{1/q} 
\leq C \int_{\R^{N}} |g(x)|v(x) \, dx, 
\end{equation}
for all $g \in {L^1_v}(\R^N,F)$ satisfying $L(D)g=0$, in the sense of distributions.
\end{lemma}

Estimates of the type \eqref{mainineq-two-weight} for $u(x):=|x|^{-N+(N-\ell)q}$ and $v(x) \equiv 1$, with $1\leq q < \frac{N}{N-\ell}$, and $g=A(D)\varphi$ for $\varphi \in C_{c}^{\infty}(\R^{N},E)$, where $A(D)$ is elliptic and canceling, were obtained by Hounie and Picon in \cite[Lemma 2.1]{HHP}. The conditions  \eqref{12a1} and \eqref{13a1} can be understood as testing conditions for the two-weighted Stein-Weiss inequality in $L^1$ for this special class of vector fields (see \cite{DNP} for more details).

\subsection{Weighted $L^1$ type estimate for Riesz transform} 
For $N\geq 2$ and $0<\ell<N$, consider the classical Riesz potential operator
\[I_\ell f(x) := \dfrac{1}{\gamma(\ell)} \int_{\R^N} \frac{f(y)}{|x-y|^{N-\ell}} \; dy.\]
Schikorra, Spector and Van Schaftingen proved in \cite[Theorem A]{SSV} that there exists a constant $C=C(N,\ell)>0$ such that
\begin{equation}\label{riesz}
\|I_{\ell}u\|_{L^{N/(N-\ell)}}\leq C \|Ru\|_{L^{1}}, \quad  \text{for all} \,\,\, u \in C_{c}^{\infty}(\R^{N}),
\end{equation}
with $Ru \in L^{1}(\R^N,\R^{N})$, where $Ru=(R_{1}u,...,R_{N}u)$ is the vector Riesz transform. 
The previous inequality  recovers the boundedness of Riesz potential operator from Hardy space $H^1(\R^N)$ to $L^{N/(N-\ell)}(\R^N)$, since the norm {$\|u\|_{H^{1}}$} is equivalent to $\|u\|_{L^{1}}+\|Ru\|_{L^{1}}$. We point out \eqref{riesz} breaks down replacing {$\|u\|_{H^{1}}$ by $\|u\|_{L^{1}}$} (see \cite{DNP}). We state the following extension:  

\begin{theorem}
Let $N \geq 2$, $0<\ell<N$ and $1\leq q <\infty$. If $\nu \in \M_{+}(\R^N)$ and $w$ satisfies \eqref{12a} and \eqref{13a} then there exists $C>0$ such that 
\begin{equation*}
\left( \int_{\R^N} \left|I_{\ell}u(x)\right|^{q} \, d\nu \right)^{1/q}   
		\leq C \|Ru\|_{L_{w}^{1}}, \quad  {\text{for all}} \,\,\, u \in C_{c}^{\infty}(\R^{N}),
\end{equation*}
with $Ru \in L_{w}^{1}(\R^N,\R^{N})$. 
\end{theorem}

The previous result recovers \cite[Theorem 6.4]{DNP} taking $d\nu:=|x|^{-\beta}dx$, $w(x):=|x|^{\alpha}$, with $\alpha+\beta>0$, $0\leq \alpha <1$ and by scaling $\displaystyle{\frac{1}{q}=1+\frac{\alpha+\beta - \ell}{N}}$. The proof is a direct consequence of the identity $R^{2}=-Id$ and the Lemma \ref{main} applied to the kernels $K_{j}$ given by $\widehat{K_{j}}(\xi):=(I_{\ell}\circ R_{j})^{\,\widehat{\,\,}}(\xi)$ that satisfy
 \eqref{kc} and \eqref{kd}, and $g:=Ru$ that belongs to the kernel of some cocanceling operator,
 since $\partial_{x_{i}}R_{j}=\partial_{x_{j}}R_{i} $ for $i \neq j$ (see \cite{DNP} for more details).

\subsection{Trace and fractional inequalities for vector fields}

Next we state the following extension of limiting trace inequality presented at \cite[Theorem 5.1]{BP}.

\begin{theorem} \label{raitathm}

Let $A(D)$ be a homogeneous linear differential operator of order m on $\R^{N}$, $N\ge2$, from $E$ to $F$. If $A(D)$ is elliptic and canceling, then for all $\nu \in \M_{+}(\R^N)$ and every weight $w$ 
satisfying {\eqref{12a} and \eqref{13a} with $\ell=1$ and $q=1$} 
there exists $C>0$ such that 
\begin{equation}\label{raita1}
\int_{\R^N} \left|D^{m-1}u(x)\right| \, d\nu   \leq C 
		 \|A(D)u\|_{L_{w}^{1}}, \quad  {\text{for all}} \,\,\, u \in C_{c}^{\infty}(\R^{N},E).
\end{equation}
\end{theorem}

\begin{proof}
Adapting the identity \eqref{decfrac}, for each $|\alpha|=m-1$ we may write  
$$D^{\alpha}u(x)=\int_{\R^{N}}K_{\alpha}(x-y)[A(D)u(y)]\,dy,$$ 
where $\widehat{K_{\alpha}}(\xi):=\xi^{\alpha}(A^{\ast}\circ A)^{-1}(\xi)A^{\ast}(\xi)$, 
which satisfies \eqref{kc} and \eqref{kd} for $\ell=1$. Thus, the estimate \eqref{raita1} follows
by Lemma \ref{main} for $q=1$, as showed in the proof of inequality \eqref{eqmain}.
\end{proof}

We can also prove the following weighted Hardy-Littlewood-Sobolev inequality associated to fractional derivatives:

\begin{theorem}
Let $A(D)$ be a homogeneous linear differential operator of order {$m$} on $\R^{N}$, $N \geq 2$, from $E$ to $F$, and assume that  $1 \leq q <\infty$, $0<\ell<N$ and $\ell \leq m$. If $A(D)$ is elliptic and canceling, then for all $\nu \in \M_{+}(\R^N)$ and every weight $w$ satisfying {\eqref{12a} and \eqref{13a}}
there exists $C>0$ such that 
\begin{equation*}
\left(\int_{\R^N} \left| (-\Delta)^{(m-\ell)/2}u(x)\right|^{q} \, d\nu \right)^{1/q}  \leq C  
		 \|A(D)u\|_{L^{1}_w}, \quad {\text{for all}} \,\,\, u \in C_{c}^{\infty}(\R^{N},E).
\end{equation*}
\end{theorem}

We omit the proof, which is similar to that of Theorem \ref{raitathm}. In particular, it recovers \cite[Theorem 5.2]{BP} if {$w \equiv1$}, and the inequality (1.5) in \cite[Theorem A]{HHP} when also taking $d\nu=|x|^{-N+(N-\ell)q}\,dx$ for $1 \leq q <N/(N-\ell)$.\\

\noindent \textbf{Acknowledgments:} The authors would like to thank Prof. Laurent Moonens for his contribution in some discussions regarding solvability in weighted Lebesgue spaces. The third author gratefully acknowledges the support and hospitality of Universit\'e Paris-Saclay (Laboratoire de Math\'ematiques dOrsay), CNRS and IHES where part of this research was developed during his visiting period.

\end{document}